\documentclass[12pt]{article}
\usepackage{latexsym}
\usepackage{amsmath}
\usepackage{amssymb}
\usepackage{amsthm}
\usepackage{graphicx}
\usepackage{hyperref}
\usepackage[latin1]{inputenc}
\usepackage[T1]{fontenc}
\usepackage[all]{xy}
\makeatletter
\renewcommand{\@seccntformat}[1]
{\csname the#1\endcsname.\enspace} \makeatother
\def \Xint#1{\mathchoice
   {\XXint\displaystyle\textstyle{#1}}%
   {\XXint\textstyle\scriptstyle{#1}}%
   {\XXint\scriptstyle\scriptscriptstyle{#1}}%
   {\XXint\scriptscriptstyle\scriptscriptstyle{#1}}%
   \!\int}
\def \XXint#1#2#3{{\setbox0=\hbox{$#1{#2#3}{\int}$}
     \vcenter{\hbox{$#2#3$}}\kern-.5\wd0}}

\def \dashint{\Xint-}
\newtheorem{theorem}{Theorem}
\newtheorem{lemma}{Lemma}
\newtheorem{remark}{Remark}

\newtheorem{example}{Example}
\newtheorem{definition}{Definition}

\begin{document}
\begin{center}
   {\bf On the distribution of extrema for a class of L\'evy processes\footnote{\today}}\\
{\sc Amir T. Payandeh Najafabadi$^{a,}$\footnote{Corresponding
author:
amirtpayandeh@sbu.ac.ir; Phone No. +98-21-29903011; Fax No. +98-21-22431649} \& Dan Kucerovsky$^{b}$}\\
a  Department of Mathematical Sciences, Shahid Beheshti
University, G.C. Evin, 1983963113, Tehran, Iran.\\
b Department of Mathematics and Statistics, University of New
Brunswick, Fredericton, N.B. CANADA E3B 5A3.
\end{center}
\begin{center}
    {\sc Abstract}
\end{center}
Suppose $X_t$ is either a \emph{regular exponential type} L\'evy
process  or a L\'evy process with a \emph{bounded variation jumps
measure}. The distribution of the extrema of $X_t$ play a crucial role in many financial
and actuarial problems. This article employs the well known and
powerful Riemann-Hilbert technique to derive the
characteristic functions of the extrema for such L\'evy processes. An
approximation technique along with several examples is given.
\section{Introduction}
Suppose $X_t$ be a one-dimensional, real-valued, right
continuous with left limits (c\`adl\`ag), and adapted L\'evy process, starting at zero. Suppose also that the corresponding jumps measure, $\nu,$
is defined on ${\Bbb R}\setminus\{0\}$ and satisfies $\int_{\Bbb
R}\min\{1,x^2\}\nu(dx)<\infty.$ Moreover, suppose the stopping time
$\tau(q)$ is either a geometric or an exponential distribution
with parameter $q$ that is independent of the L\'evy process $X_t$, and that
$\tau(0)=\infty.$ The extrema of the L\'evy process $X_t$ are defined to be
\begin{eqnarray}
\label{definition-extrema}
\nonumber  M_q &=& \sup\{X_s:~s\leq\tau(q)\};\\
           I_q &=& \inf\{X_s:~s\leq\tau(q)\}.
\end{eqnarray}
The Wiener-Hopf factorization method is a technique that can be
used to study the characteristic function of $M_q$ and $I_q.$ The
Wiener-Hopf method has been used to show that:
\begin{description}
    \item[(i)] The random variables $M_q$ and $I_q$ are independent
(Kuznetsov; 2009b and Kypriano; 2006 Theorem 6.16);
    \item[(ii)] The product of their characteristic
functions is equal to the characteristic function of the L\'evy
process $X_t$ (Bertoin; 1996 page 165); and
    \item[(iii)] The random
variable $M_q$ ($I_q$) is infinitely divisible, positive
(negative), and has zero drift (Bertoin; 1996, page 165).
\end{description}
Supposing that the characteristic function of a L\'evy process,
$X_t,$  can be decomposed as a product of two functions, one of
which is the boundary values of a of a function that is  analytic
and bounded in the  complex upper half-plane (i.e., ${\Bbb
C}^+=\{\lambda:~\lambda\in{\Bbb C}~\hbox{and}~Im(\lambda)\geq0\}$)
and the other of which is the boundary values of a function that
is analytic and bounded in the complex lower half-plane
(i.e.,${\Bbb C}^-=\{\lambda:~\lambda\in{\Bbb
C}~\hbox{and}~Im(\lambda)\leq0\}$), we then have that the
characteristic functions of $M_q$ and $I_q$
 can be determined explicitly. The required decomposition can be obtained explicitly if, for example,
 the characteristic function of the L\'evy process is a rational function. Furthermore,
 there is a very general existence result for such decompositions, based on the theory
 of singular integrals (specifically Sokhotskyi-Plemelj integrals). Lewis \&
Mordecki (2005) considered a L\'evy process $X_t$ which has
negative jumps distributed according to a mixed-gamma family of
distributions and has an arbitrary positive jumps measure. They
established that such a process has a characteristic function
which can be decomposed as a product of a rational function and a
more or less arbitrary function, and that these functions are
analytic in ${\Bbb C}^+$ and ${\Bbb C}^-,$ respectively. Recently,
they provided an analogous result for a L\'evy process whose
positive jumps measure is given by a mixed-gamma family of
distributions and whose negative jumps measure has an arbitrary
distribution, more detail can be found in Lewis \& Mordecki
(2008).

Unfortunately, in the majority of situations, the characteristic function
of the process is not  a rational function nor can
be explicitly decomposed as a product of two analytic functions in ${\Bbb
C}^+$ and ${\Bbb C}^-.$ Of course, there is a general theory allowing the
characteristic functions of $M_q$ and $I_q$ to be expressed in terms of a
Sokhotskyi-Plemelj integral (see Equation \ref{Plemelj-integral}).
This provides an existence result, but presents some difficulties in
numerical work due to
slow evaluation and numerical problems caused by singularities in the complex plane that are
near the contour used in the integral. To overcome these problems,
approximation methods may be  considered.

Roughly speaking, the Wiener-Hopf factorization technique attempts
to find a function $\Phi$ that is analytic, bounded, and
complex-valued except for a prescribed jump discontinuity on the
real line within the complex plane. The radial limits at the real
line, denoted $\Phi^\pm,$ satisfy
$\Phi^+(\omega)\Phi^-(\omega)=g(\omega),$ where $\omega\in{\Bbb
R}$ and $g$ is a given function with certain conditions ($g$ is a
zero index function which satisfies the H\"older condition). The
radial limits provide the desired decomposition of $g$ into a
product of boundary value functions that was alluded to above. The
Wiener-Hopf factorization technique can be extended to a more
general setting and is then also known as the Riemann-Hilbert
method. The Riemann-Hilbert method is theoretically well developed
and it is often more convenient to work with than the Wiener-Hopf
technique, see Kucerovsky \& Payandeh (2009) for more detail. The
Riemann-Hilbert problem has proved remarkably useful in solving an
enormous variety of model problems in a wide range of branches of
physics, mathematics, and engineering. Kucerovsky, et al. (2009)
employed the Riemann-Hilbert problem to solve a statistical
decision problem. More precisely, using the Riemann-Hilbert
problem, they established the mle estimator under absolute-error
loss function is a generalized Bayes estimator for a wide class of
location family of distributions.

This article considers the problem of finding the distributions of the extrema of a
L\'evy process whose ({\bf i}) \emph{either} its corresponding
jumps measure is a finite variation measure \emph{or} is the
regular exponential L\'evy type process.; and ({\bf ii}) its
corresponding stopping time $\tau(q)$ is either a geometric or an
exponential distribution with parameter $q$ independent of the
L\'evy process $X_t$ where $\tau(0)=\infty.$ Then, it develops a
procedure in terms of the well known and  powerful
Riemann-Hilbert technique, to solve the problem of finding the characteristic functions of $M_q$
and $I_q$. A remark has been made that is helpful in the
situation where such  characteristic functions cannot be found
explicitly.

Section 2 collects some essential elements which are required for
other sections. Section 3 states the problem of finding the
 characteristic functions for the distribution of the extrema in terms of a Riemann-Hilbert
problem. Then, in that section is derived an expression for such characteristic
functions in terms of the Sokhotskyi-Plemelj integral. A remark that is helpful in
situations where such characteristic functions cannot be found
explicitly is made, and several examples are given.
\section{Preliminaries}
Now, we collect some lemmas which are used later.

The index of an analytic function $h$ on ${\Bbb R}$ is the number
of zero minus number of poles of $h$ on ${\Bbb R},$ see Payandeh
(2007, chapter 1), for more technical detail. Computing the index
of a function is usually a \emph{key step} to determine the
existence and number of solutions of a Riemann-Hilbert problem. We
are primarily interested in the case of zero index.

The \emph{Sokhotskyi-Plemelj integral} of a function $s$ which
satisfies the H\"older condition and it is defined by a principal
value integral, as follows.
\begin{eqnarray}
\label{Plemelj-integral}\phi_s(\lambda ):=\frac 1{2\pi
i}\dashint_{ {\Bbb R}}\frac {s(x)}{x -\lambda}dx ,~~\hbox{for}~
\lambda\in {\Bbb C}.\end{eqnarray} The following are some well
known properties of the Sokhotskyi-Plemelj integral, proofs can be
found in Ablowitz \& Fokas (1990, chapter 7), Gakhov (1990,
chapter 2), and Pandey (1996, chapter 4), among others.
\begin{lemma}
\label{Sokhotskyi-Plemelj-properties} The radial limit of the
Sokhotskyi-Plemelj integral of $s,$   given by
$\phi^{\pm}_s(\omega )=\displaystyle\lim_{\lambda\rightarrow
\omega +i0^{\pm}}\phi_s(\lambda )$ can be represented as:
\begin{description}
    \item[i)] jump formula. i.e., $\phi^{\pm}_s(\omega )=\pm s(\omega
)/2+\phi_s(\omega ),$ where $\omega\in {\Bbb R};$
    \item[ii)] $\phi^{\pm}_s(\omega )=\pm s(\omega )/2+H_s(\omega )/(2i),$ where
$H_s(\omega)$ is the \emph{Hilbert transform} of $s$ and
$\omega\in {\Bbb R}.$
\end{description}
\end{lemma}
The Riemann-Hilbert problem is the function-theoretical problem of
finding a single function which is analytic separately in ${\Bbb
C}^+$ and ${\Bbb C}^-$ (called sectionally analytic) and having a
prescribed jump discontinuity on the real line. The following
states the homogeneous Riemann-Hilbert problem which one deals
with in studying the characteristic functions of $M_q$ and $I_q$.
\begin{definition}
\label{Riemann-Hilbert-problem} The homogeneous Riemann-Hilbert
problem, with zero index,  is the problem of  finding a sectionally analytic function
$\Phi$ whose the upper and lower radial limits at the real line,
$\Phi^{\pm},$ satisfy
\begin{equation}
\label{equation-of-Riemann-Hilbert-problem} \Phi^{
+}(\omega)=g(\omega )\Phi^{-}(\omega),~~\hbox{for}~w \in {\Bbb
R},\end{equation} where  $g$ is a given continuous function
  satisfying a H\"older condition on ${\Bbb R}.$ Moreover, $g$
is assumed to have zero index, to be non-vanishing on ${\Bbb R},$
and bounded above by 1.
\end{definition}
A homogeneous Riemann-Hilbert problem always has a family of
solutions if no restrictions on growth at infinity are posed. But
a unique solution can be obtained with further restrictions.
Solutions vanishing at infinity are the most common restriction
considered in mathematical physics and in engineering
applications, see Payandeh (2007, chapter 1), for more detail.
With these restrictions, the solutions of the homogeneous
Riemann-Hilbert problem are given by
\begin{eqnarray*}
  \Phi^{\pm}(\lambda ) &=&\exp\{\pm\phi_{\ln (g)}(\lambda)\},~\hbox{for~}\lambda\in {\Bbb C}
\end{eqnarray*}
where $\phi_{\ln (g)}$ stands for the Sokhotskyi-Plemelj integral,
given by \ref{Plemelj-integral}, of $\ln (g).$

In this paper, we need to solve a homogeneous Riemann-Hilbert
problem (also  known as a Wiener-Hopf factorization problem) with
\begin{eqnarray}
\label{RH-For-this-paper}
  \Phi^+(\omega)\Phi^-(\omega) &=& g(\omega),~\omega\in{\Bbb R},
\end{eqnarray}
where $g$ is a given, zero index function which satisfies the
H\"older condition and $g(0)=1.$ For convenience in presentation,
 we will simply call the above homogeneous Riemann-Hilbert
problem a Riemann-Hilbert problem. The following provides
solutions for the above Riemann-Hilbert problem. We begin with what we term the Resolvent Equation for  Sokhotskyi-Plemelj integrals.
\begin{lemma} The {\it Sokhotskyi-Plemelj} integral of a function
$f$ satisfies
$$\phi_f (\lambda) - \phi_f (\mu) = (\lambda-\mu)\phi_{\frac{f(x)}{x-\lambda}}(\mu),$$
for $\lambda$ and $\mu$ real or complex. \label{lem:resolvent}
\end{lemma}
\begin{proof} In general, $$(x-\lambda)^{-1} - (x-\mu)^{-1} = (\lambda-\mu)(x-\mu)^{-1}(x-\lambda)^{-1}.$$
Then, see Dunford \& Schwartz (1988), we have an equation of
Cauchy integrals, where $\Gamma=\Bbb R$:
$$\frac{1}{2\pi i} \int_\Gamma \frac{f(x)}{x-\lambda}dx - \frac{1}{2\pi i} \int_\Gamma \frac{f(x)}{x-\mu}dx = \frac{\lambda-\mu}{2\pi i} \int_\Gamma \frac{f(x)}{(x-\mu)(x-\lambda)}dx .$$
\end{proof}
The above is valid only for $\lambda$ and $\mu$ not on the real
line. However, by Lemma \ref{Plemelj-integral} the values of
$\phi_f$ on the real line are obtained by averaging the limit from
above, $\phi^+_f$, and the limit from below, $\phi^{-}_f.$ We thus
obtain the stated equation in all cases.
\begin{lemma}
\label{Solution-RH-For-Our-Paper} Suppose $\Phi^{\pm}$ are
sectionally analytic functions satisfying the Riemann-Hilbert problem
given by \ref{RH-For-this-paper}. Moveover, suppose that $g$ is a
zero index function satisfies the H\"older condition and $g(0)=1.$
Then,
\begin{eqnarray*}
\Phi^\pm(\lambda)&=&\exp \{\pm\phi_{\ln g}(\lambda)\mp\phi_{\ln
g}(0)\},~\lambda\in{\Bbb C}
\end{eqnarray*}
where $\phi_{\ln g}$ stands for the Sokhotskyi-Plemelj integration
of $\ln g.$
\end{lemma}
\textbf{Proof.} By taking logarithm from  both sides, the above
equation can be rewritten as
\begin{eqnarray*}
  \ln \Phi^+(\omega)-(-\ln \Phi^-(\omega)) &=& \ln g(\omega).
\end{eqnarray*}
Since $\ln g(0)=0,$ the above  equation does not satisfy the
non-vanishing condition  of the standard Riemann-Hilbert problem.
One may handle this  by dividing both sides by $\omega$ (Gakhov;
1990) suggested this kind of modification to extend the domain of
the Riemann-Hilbert method). Now, we have
\begin{eqnarray*}
  \frac{\ln \Phi^+(\omega)}{\omega}-\frac{(-\ln \Phi^-(\omega))}{\omega} &=& \frac{\ln
  g(\omega)}{\omega}.
\end{eqnarray*}
 The above equation
meets all conditions for the usual solution of the additive Riemann-Hilbert problem by
Sokhotskyi-Plemelj integrals, and therefore, the
solutions of our Riemann-Hilbert problem (equation \ref{RH-For-this-paper}) are
\begin{eqnarray*}
  \Phi^\pm(\lambda) &=& \exp\{\pm\frac{\lambda}{2\pi i}\dashint_{ {\Bbb R}}\frac {\ln g(x)/x}{x
  -\lambda}dx\},~\lambda\in{\Bbb C}.
\end{eqnarray*}
Lemma \ref{lem:resolvent} with $f=\ln g$ gives
 $$\phi_{\ln g} (\lambda) - \phi_{\ln g} (\mu) = (\lambda-\mu)\phi_{\frac{\ln g(x)}{x-\lambda}}(\mu).$$
Letting $\lambda$ go to zero from above, in the complex plane, and using the fact that $\ln g(0) = 0$,  Lemma \ref{Plemelj-integral} lets us conclude that
 $$\phi_{\ln g} (0) - \phi_{\ln g} (\mu) = -\mu\phi_{\frac{\ln g(x)}{x}}(\mu).$$
 Substituting this into the above equation for $\Phi^\pm$ gives our claimed
 result.~$\square$

The following explores some properties of the above lemma.
\begin{remark}
\label{solutions-of-RH-in-term-g} Using the jump formula. One can
conclude that
\begin{eqnarray*}
\Phi^\pm(\omega)&=&\sqrt{g(\omega)}\exp \{\pm\frac{i}{2}(H_{\ln
g}(0)-H_{\ln g}(\omega))\},
\end{eqnarray*}
where $H_{\ln g}$ stands for the Hilbert transform of $\ln g.$
\end{remark}
The following explores Carlemann's technique for obtaining solutions of
the Riemann-Hilbert problem \ref{RH-For-this-paper} directly
rather than using the Sokhotskyi-Plemelj integrations.
\begin{remark} (Carlemann's technique)
\label{Carlemann-method} If $g$ can be decomposed as a product of
two sectionally analytic  functions $g^+$ and $g^-,$ respectively
in ${\Bbb C}^+$ and ${\Bbb C}^-.$ Then, solutions of the
Riemann-Hilbert problem \ref{RH-For-this-paper} are $\Phi^+\equiv
g^+$ and $\Phi^-\equiv g^-.$
\end{remark}
Carlemann's method amounts to solution by inspection.
The most favorable situation for the Carlemann's method is the
case where $g$ is a rational function. In the case that approximation
solutions are required, Kucerovsky \& Payandeh (2009) suggested approximating $g$ with a rational function obtained from a Pad\'e
approximant or a continued fraction expansion.

The Paley-Wiener theorem is one of the key elements for restating
problem of finding the characteristic functions of extrema in a
L\'evy process as a Riemann-Hilbert problem, as in equation
(\ref{RH-For-this-paper}). The theorem is stated below, a proof
may be found in Dym \& Mckean (1972, page 158).
\begin{theorem} (Paley-Wiener)
\label{Paley.Wiener} Suppose $s$ is a function in $L^2({\Bbb R}),$
then the following are equivalent:
\begin{description}
     \item[i)] The real-valued function $s$ vanishes on the left half-line.
     \item[ii)] The Fourier transform $s$, say, $\hat {s}$ is holomorphic on $
{\Bbb C}^{+}$ and the $L^2({\Bbb R})$-norms of the functions
$x\mapsto\hat {s}(x+iy_0)$ are uniformly bounded for all $y_0\geq
0.$
\end{description}
\end{theorem}

\begin{definition} (Mixed gamma family of distributions)
\label{mixed gamma} A nonnegative random variable $X$ is said to
be distributed according to a mixed gamma distribution if its
density function is given by
\begin{eqnarray}
\label{mixed gamma-density}
  p(x) &=&
  \sum_{k=1}^{\nu}\sum_{j=1}^{n_\nu}c_{kj}\frac{\alpha_k^jx^{j-1}}{(j-1)!}e^{-\alpha_k
  x},~x\geq0,
\end{eqnarray}
where $c_{k_j}$ and $\alpha_k$ are positive value where
$\sum_{k=1}^{\nu}\sum_{j=1}^{n_\nu}c_{k_j}=1.$
\end{definition}
The following explores some properties of the characteristic
function of the above, a proof can be found in Bracewell (2000,
page 433), and Lewis \& Mordecki (2005), among others.
\begin{lemma}
\label{properties-characteristic-function} The characteristic
function of a distribution (or equivalently the Fourier transform
of its density function), say ${\hat p},$ has the following
properties:
\begin{description}
    \item[i)] ${\hat p}$ is a rational function if and only if the density function belongs to the class of mixed gamma
    family of distributions, given by \ref{mixed gamma-density};
    \item[ii)] ${\hat p}(\omega)$ is a Hermition function, i.e., the real
    part of ${\hat p}$ is even function and the imaginary part odd function;
    \item[iii)] ${\hat p}(0)=1;$ and the norm of ${\hat p}(\omega)$ bounded by
    1.
\end{description}
\end{lemma}
\section{Main results}
Suppose that $X_t$ is a one-dimensional real-valued L\'evy process
starting at $X_0=0$ and defined by a triple $(\mu,\sigma,\nu):$
the drift $\mu\in{\Bbb R},$ volatility $\sigma\geq0,$ and the
jumps measure $\nu$ is given by a nonnegative function defined on
${\Bbb R}\setminus\{0\}$ satisfying $\int_{\Bbb
R}\min\{1,x^2\}\nu(dx)<\infty.$ The L\'evy-Khintchine representation
states that the characteristic exponent $\psi$ (i.e.,
$\psi(\omega)=\ln (E(\exp(i\omega X_1))),~\omega\in{\Bbb R}$) can
be represented by
\begin{eqnarray}
\label{Levy-Khintchine} \psi(\omega) &=&
i\mu\omega-\frac{1}{2}\sigma^2\omega^2+\int_{{\Bbb R}}(e^{i\omega
x}-1-i\omega xI_{[-1,1]}(x))\nu(dx),~~\omega\in{\Bbb R}.
\end{eqnarray}
Now, we explore some properties of the two expressions
$q(q-\psi(\omega))^{-1}$ and $(1-q)(1-q\psi(\omega))^{-1},$
$\omega\in{\Bbb R}$ that will play an essential r\^ole in the rest of
this section.
\begin{lemma}
\label{Holder-condition} The L\'evy process $X_t$ has a jumps
measure $\nu$ that satisfies $\int_{{\Bbb
R}\setminus[-1,1]}|x|^\varepsilon v(x)<\infty,$ for some
$\varepsilon\in(0,1).$ Then
\begin{description}
    \item[i)] $q(q-\psi(\omega))^{-1}$ satisfies the H\"older
    condition;
    \item[ii)] $(1-q)(1-q\exp\{-\psi(\omega)\})^{-1}$ satisfies
the H\"older condition.
\end{description}
\end{lemma}
\textbf{Proof.} A proof of part (i) may be found in Kuznetsov
(2009a) and the  proof of part (ii) is a minor variation of the
proof of part (i). $\square$

The above condition on the jumps measure $v,$ (\textit{i.e.},
$\exists\varepsilon\in(0,1),~\mbox{such that}~ \int_{{\Bbb
R}\setminus[-1,1]}|x|^\varepsilon v(x)<\infty$) is a very mild
restriction and many L\'evy processes, such as all stable
processes, meet this property. It only excludes cases which the
jumps measure has extremely heavy tail (behaves like
$|x|^{-1}/(\ln|x|)^2$ for large enough $x$), see Kuznetsov (2009a)
for more detail.

The following recalls the definition of a very useful class of
L\'evy processes.
\begin{definition}
\label{regular-exponential-type} A L\'evy process $X_t$ is said to
be of regular exponential type, if its corresponding characteristic
exponent is analytic and continuous in a strip about the real
line.
\end{definition}
Loosely speaking, a L\'evy process $X_t$ is a regular L\'evy
process of exponential type (RLPE) if its jumps measure has a
polynomial singularity at the origin and decays exponentially at
infinity, see Boyarchenko \& Levendorski\u{l} (1999, 2002a-c). The
majority of classes of L\'evy processes used in empirical studies
in financial markets satisfy the conditions given above (i.e., $\psi$
is analytic and continuous in a stripe about the real line).
Brownian motion, Kou's model (Kou, 2002); hyperbolic processes
(Eberlein \& Keller, 1995, Eberlein et al, 1998, and
Barndorff-Nielsen, et at, 2001); normal inverse gaussian processes
and their generalization (Barndorff-Nielsen, 1998 and
Barndorff-Nielsen \& Levendorski\u{l} 2001); extended Koponen's
family (Koponen, 1995 and Boyarchenko \& Levendorski\u{l}, 1999)
are examples of the regular L\'evy process of exponential type.
While the variance gamma processes (Madan et al, 1998) and stable
L\'evy processes are two important exceptions, see Cardi (2005)
for more detail.
\begin{lemma}
\label{psi-analytic-bounded} The L\'evy process $X_t$ has a analytic
and continuous characteristic exponent on the real line ${\Bbb R}$
\emph{either} one the following conditions are hold:
\begin{description}
    \item[i)] $X_t$ has a jumps measure $\nu(dx)$ with bounded variation  (\textit{i.e.}, $\int_{-1}^1x\nu(dx)<\infty$).
    \item[ii)] $X_t$ is a regular L\'evy process of exponential type (see Definition
    \ref{regular-exponential-type}).
\end{description}
\end{lemma}
\textbf{Proof.} For part (i), observe that the characteristic
exponent for the bounded variation jumps measure $\upsilon(dx)$ is
given by
\begin{eqnarray*}
  \psi(\omega) &=& i\mu\omega-\int_{{\Bbb R}}(e^{i\omega x}-1-i\omega
  xI_{[-1,1]}(x))\upsilon(dx)\\
  &=& i\mu\omega-1-i\omega\int_{[-1,1]}x\upsilon(dx)+\int_{(-\infty,0]}e^{i\omega
  x}\upsilon^-(dx)+\int_{(0,\infty)}e^{i\omega x}\upsilon^+(dx),
\end{eqnarray*}
see Bertoin (1996). From the fact that $\upsilon(dx)$ is a bounded
variation jumps measure, one can conclude that three first terms
are analytic on ${\Bbb R}.$ A double application of the
Paley-Wiener Theorem \ref{Paley.Wiener} shows that two last terms
are, respectively, analytic and bounded in ${\Bbb C}^-$ and ${\Bbb
C}^+.$ Therefore, these terms are analytic on ${\Bbb R}={\Bbb
C}^-\cap{\Bbb C}^+.$ The proof of part (ii) follows from Definition
\ref{regular-exponential-type}. $\square$

\begin{lemma}
\label{index-zero} Suppose the L\'evy process $X_t$ \emph{either}
is a regular exponential type \emph{or} has a bounded variation
jumps measure $\nu.$ Then,
\begin{description}
    \item[i)] letting the geometric stopping time be $\tau(q),$ with parameter $q$ ($q\neq1$), the function
    $(1-q)(1-q\exp\{-\psi(\omega)\})^{-1}$ has zero index on the real line;
    \item[ii)] for exponential stopping time $\tau(q)$ with constant rate $q$ ($q>0$), the function
    $(q)(q-\psi(\omega))^{-1}$ has zero index on the real line.
\end{description}
\end{lemma}
\textbf{Proof.} Firstly, observe that the functions
$q(q-\psi(\omega))^{-1}$ and
$(1-q)(1-q\exp\{-\psi(\omega)\})^{-1}$ have no zero on ${\Bbb R}.$
They may have a zero at $\pm\infty.$ Moreover, equations
$q-\psi(\omega)=0$ and $1-q\exp\{-\psi(\omega)\}=0$ are,
respectively, equivalent to $E(\exp\{-i\omega X_1\})=\exp\{q\}$
and $E(\exp\{-i\omega X_1\})=q.$ Since $q$ is positive, real
valued, and $E(\exp\{-i\omega X_1\})$ is a Hermitian function,
these equations have no solutions on ${\Bbb R}.$ Moreover,  from
Lemma (\ref{psi-analytic-bounded}) observe that two functions
$q(q-\psi(\omega))^{-1}$ and
$(1-q)(1-q\exp\{-\psi(\omega)\})^{-1}$ are analytic and bounded on
the real line. The desired proof comes from the above observations
along with the fact that the index of an analytic function is the
number of zeros minus number of poles within the contour (Gakhov;
1990). $\square$

The extrema of a L\'evy process play a crucial role in determining
many aspects of a L\'evy process, see Mordecki (2003), Renming \&
Vondra\v{c}ek (2008), Dmytro (2004), and Albrecher, et al. (2008),
among many others.

The following theorem addresses the question of how the problem of finding the
characteristic functions of the distribution of the extrema can be restated in term of a
Riemann-Hilbert problem \ref{RH-For-this-paper}.
\begin{theorem}
\label{Exact-distributions} Suppose $X_{t}$ is a L\'evy process
whose stopping time $\tau(q)$ has either a geometric or an
exponential distribution with parameter $q$ independent of the
L\'evy process $X_t$ and $\tau(0)=\infty.$ Moreover, suppose that
\begin{description}
    \item[$A_1$)] its jumps measure $\nu$ satisfies $\int_{{\Bbb
    R}\setminus[-1,1]}|x|^\varepsilon\nu{dx}<\infty,$ for some $\varepsilon>0;$
    \item[$A_2$)] either its jumps measure $\nu$ is of bounded variation or
    $X_t$ is a regular exponential type L\'evy process.
\end{description}
Then, the characteristic functions of $M_q$ and $I_q,$, say  $\Phi^+_q$ and $\Phi^-_q,$ respectively, satisfy
\begin{description}
    \item[i)] the Riemann-Hilbert problem $\Phi^+_q(\omega)\Phi^-_q(\omega)=q(q-\psi(\omega))^{-1},~\omega\in{\Bbb R},$ whenever
    $\tau(q)$ has an exponential distribution with parameter $q$ ($q>0$). has a unique solution
    $$\Phi_q^\pm(\omega)=\sqrt{q/(q-\psi(\omega))}\exp \{\pm\frac{i}{2}(H_{\ln
(q-\psi)}(\omega)-H_{\ln (q-\psi)}(0)\},~\omega\in{\Bbb R};$$
    \item[ii)] the Riemann-Hilbert problem $\Phi^+_q(\omega)\Phi^-_q(\omega)=(1-q)(1-q\psi(\omega))^{-1},~\omega\in{\Bbb R},$ whenever
    $\tau(q)$ has a geometric distribution with parameter $q$ ($q\neq1$), has a unique solution $$\Phi_q^\pm(\omega)=\sqrt{(1-q)/(1-q\exp\{-\psi(\omega)\})}\exp \{\pm\frac{i}{2}(H_{\ln
(1-qe^{-\psi})}(\omega)-H_{\ln
(1-qe^{-\psi})}(0)\},~\omega\in{\Bbb R}.$$
\end{description}
\end{theorem}
\textbf{Proof.} To establish the desired result observe that:
({\bf 1}) the characteristic function of the L\'evy process $X_t$
can be uniquely decomposed as a product of two characteristic
functions of the {\it supremum} and {\it infimum} of the process,
see Cardi (2005, pages 43--4), for more detail; ({\bf 2}) the
functions $M_q$ and $I_q$ attain, respectively, nonnegative and
nonpositive values. Therefore, a double application of the
Paley-Wiener theorem (Theorem \ref{Paley.Wiener}) shows $\Phi^+_q$
and $\Phi^-_q,$ respectively, are sectionally analytic in ${\Bbb
C}^+$ and ${\Bbb C}^-;$ ({\bf 3}) The two expressions
$q(q-\psi(\omega))^{-1}$ and
$(1-q)(1-q\exp\{-\psi(\omega)\})^{-1}$ satisfy a H\"older
condition (see Lemma \ref{Holder-condition}) and have zero index
(see Lemma \ref{index-zero}); ({\bf 4}) The characteristic
function of the L\'evy process $X_t$ is $q(q-\psi(\omega))^{-1},$
for an exponential distribution stopping time $\tau(q)$ (see
Cardi; 2005, page 26) \emph{and}
$(1-q)(1-q\exp\{-\psi(\omega)\})^{-1},$ for a geometric stopping
$\tau(q)$ (see Cardi; 2005, page 25). The above observations along
with Remark \ref{solutions-of-RH-in-term-g} complete the proof.
$\square$

The following examples provide application of the above results
for several L\'evy processes.
\begin{example}
\label{from-paper-lewis-mordecki-positive-jumps} Lewis \& Mordecki
(2008) considered L\'evy process $X_t$ with an exponential
stopping time $\tau(q)$ and a jumps measure $\nu$ given by
$\nu(dx)=\nu^-(dx)I_{(-\infty,0)}(x)+\lambda
  p(x)I_{(0,\infty)}(x)dx$
where $p$ is the mixed gamma density function given by Equation
\ref{mixed gamma-density} with $0<\alpha_1<
Re(\alpha_2)\leq\cdots\leq Re(\alpha_v).$ They established that an
expression $q(q-\psi(\lambda))^{-1},$ (for $\lambda\in{\Bbb C}$):
(i) has zeros at $i\alpha_1,i\alpha_2,\cdots,i\alpha_v,$
respectively, with order $n_1,n_2,\cdots,n_v$ in ${\Bbb C}^-$ (ii)
has poles at $i\beta_1(q),i\beta_2(q),\cdots,i\beta_\mu(q),$
respectively, with multiplicities $m_1(q),m_2(q),\cdots,m_\mu(q)$
in ${\Bbb C}^-.$ Using these observations, one may decompose an
expression $q(q-\psi(\lambda))^{-1},~\lambda\in{\Bbb C},$ as a
product of two analytic in ${\Bbb C}^+$ and ${\Bbb C}^-,$ say
respectively, $\rho^+_q$ and $\rho^-_q,$ i.e.,
$q(q-\psi(\lambda))^{-1}=\rho^+_q(\lambda)\rho^-_q(\lambda),$
where
$\rho_q^+(\lambda)=q(q-\psi(\lambda))^{-1}\prod_{j=1}^{\mu(q)}(\lambda-i\beta_j(q))^{m_j(q)}\prod_{k=1}^{v}(\lambda-i\alpha_k)^{-n_k},$
$\rho_q^-(\lambda)=\prod_{k=1}^{v}(\omega-i\alpha_k)^{n_k}\prod_{j=1}^{\mu(q)}(\omega-i\beta_j(q))^{-m_j(q)},$
and $\lambda\in{\Bbb C}.$ Now using Remark \ref{Carlemann-method},
one may verify Lewis \& Mordecki (2008)'s finding which
$\Phi_q^\pm\equiv\rho_q^\pm.$
\end{example}
Similar results have been established for L\'evy process $X_t$
which has a mixed gamma negative jumps measure and an arbitrary
positive jumps, see Lewis \& Mordecki (2005) for more details.
\begin{example}
\label{alpha-stable-processes} Consider the $\alpha-$stable
process $X_t$ having an exponential stopping time $\tau(q)$ and a
jumps measure
$\nu(dx)=c_1x^{-1-\alpha}I_{(0,\infty)}(x)dx+c_2|x|^{-1-\alpha}I_{(-\infty,0)}(x)dx,$
where  $\alpha\in(0,1)\cup(1,2).$ Doney (1987) studied
distribution of $M_q,$ and $I_q.$ Since, the characteristic
exponent of the process is
$\psi(\omega)=(c_1+c_2)|\omega|^\alpha\{(c_1+c_2)-i(c_1-c_2)\hbox{sgn}(\omega)\tan(\pi\alpha/2)\}+i\omega\eta,$
where $\eta$ is a real-valued constant, and $\omega\in{\Bbb R}.$
An expression $q(q-\psi(\lambda))^{-1},~\lambda\in{\Bbb C},$ is a
rational function. Therefore, one readily can be found two
rational functions $\rho^+_q$ and $\rho^-_q$ which are analytic,
respectively, in ${\Bbb C}^+$ and ${\Bbb C}^-$ and
$q(q-\psi(\lambda))^{-1}=\rho^+_q(\lambda)\rho^-_q(\lambda)~\lambda\in{\Bbb
C}.$ Therefore, $\Phi^\pm_q\equiv\rho^\pm_q,$ which verifies
Doney's observation.
\end{example}
The following remark suggests an approximation technique to find
the characteristic functions of $M_q$ and $I_q,$ approximately,
whenever they cannot be found explicitly.
\begin{remark}
In the situation where function $q/(q-\psi(\omega))$ (or
$(1-q)(1-q\exp\{-\psi(\omega)\})^{-1}$) cannot be explicitly
decompose as a product of two sectionally analytic functions in
${\Bbb C}^+$ and ${\Bbb C}^-,$ we suggest to replace such function
by a rational function which is obtained from a Pad\'e approximant
or a continued fraction expansion and uniformly converges to the
original function. An application of  Carlemann's method leads to
an approximation solution for the characteristic functions of
$M_q$ and $I_q.$
\end{remark}
The following example represents a situation where the
characteristic functions of $M_q$ and $I_q$ apparently cannot be found
explicitly.
\begin{example}
\label{from-paper-Kuznetsov-no-1}Kuznetsov (2009b) considered a
compound Poisson process with a jumps measure
$\nu(dx)=\exp\{\alpha x\}sech(x)dx$ and an exponential stopping
time $\tau(q).$ He showed the characteristic exponent for such
compound Poisson is given by
\begin{eqnarray*}
  \psi(\omega) &=&
  \frac{\pi}{\cos(\pi\alpha/2)}-\frac{\pi}{\cosh(\pi(\omega-i\alpha)/2)},~
  \omega\in{\Bbb R}.
\end{eqnarray*}
He established that, in ${\Bbb C},$ an expression
$q(q-\psi(\cdot))^{-1}$ can be, uniformly, approximated by product
$\rho^+_q(\cdot)\rho^-_q(\cdot),$ where
\begin{eqnarray*}
  \rho^+_q(\lambda) &=&\prod_{n=0}^\infty
  \frac{(1-\frac{i\lambda}{4n+1-\alpha})(1-\frac{i\lambda}{4n+3-\alpha})}{(1-\frac{i\lambda}{4n+\eta-\alpha})(1-\frac{i\lambda}{4n+4-\eta-\alpha})};\\
  \rho^-_q(\lambda) &=&\prod_{n=0}^\infty
\frac{(1+\frac{i\lambda}{4n+1+\alpha})(1+\frac{i\lambda}{4n+3+\alpha})}{(1+\frac{i\lambda}{4n+\eta+\alpha})(1+\frac{i\lambda}{4n+4-\eta+\alpha})},
\end{eqnarray*}
where $\lambda\in{\Bbb C}$ and
$\eta=2/\pi\arccos(\pi/(q+\pi\sec(\alpha\pi/2))).$ Therefore,
approximate solutions for $\Phi^\pm_q$ are $\rho^\pm_q,$ more
detail can be found in Kuznetsov (2009b).
\end{example}
\section*{Acknowledgements}
The support of Natural Sciences and Engineering Research Council
(NSERC) of Canada are gratefully acknowledged by Kucerovsky. The
authors would like to thank Dr Nabiei for her useful discussion on
Section 3, and professor Lewis for his useful comments on
presentation of the results. Thanks to an anonymous reviewer for
constructive comments.

\end{document}